\pdfoutput=1

\documentclass{amsart}%

\usepackage[mathscr]{eucal}%
\usepackage{amssymb}%
\usepackage[usenames,dvipsnames]{xcolor} %
\usepackage[normalem]{ulem}%
\usepackage{amsthm}%
\usepackage{bbold}%
\usepackage{comment}%
\usepackage[perpage]{footmisc} %
\usepackage{enumitem}%
\usepackage{amsmath}%
\usepackage{tikz}
\usepackage{tikz-cd}
\usetikzlibrary{arrows}
\usepackage{etoolbox} 
\usepackage[unicode]{hyperref} %

\definecolor{dark-red}{rgb}{0.5,0.15,0.15}
\definecolor{dark-blue}{rgb}{0.15,0.15,0.6}
\definecolor{dark-green}{rgb}{0.15,0.6,0.15}

\hypersetup{
    colorlinks, linkcolor=Blue,
    citecolor=Blue, urlcolor=Blue
}

\usepackage{microtype}
\usepackage{quiver}
\usepackage[nameinlink,capitalise,noabbrev]{cleveref}

\numberwithin{equation}{section}%
\usepackage[all]{xy}

\xyoption{line}
\newdir{ >}{{}*!/-10pt/\dir{>}}
\usepackage{graphicx}
\usepackage{mathtools}

\swapnumbers %

\newtheorem{Thm}[equation]{Theorem}
\newtheorem*{Thm*}{Theorem}

\newtheorem{Cor}[equation]{Corollary}

\newtheorem*{Que*}{Question}

\theoremstyle{remark}
\newtheorem{Def}[equation]{Definition}

\newtheorem{Exa}[equation]{Example}

\newtheorem{Conv}[equation]{Convention}

\newtheorem{Rem}[equation]{Remark}

\tikzset{
    labelrotatebelow/.style={anchor=north, rotate=90, inner sep=1.0mm}
}
\tikzset{
    labelrotateabove/.style={anchor=south, rotate=90, inner sep=1.0mm}
}

\usetikzlibrary{decorations.markings}
\tikzset{negated/.style={
        decoration={markings,
            mark= at position 0.5 with {
                \node[transform shape] (tempnode) {$\backslash \! \! \backslash$};
            }
        },
        postaction={decorate}
    }
}

\newcommand{\nc}{\newcommand}
\nc{\dmo}{\DeclareMathOperator}

\renewcommand{\emptyset}{\varnothing}
\let\origfootnote\footnote
\newcommand{\punctfootnote}[1]{\kern-.20em\origfootnote{#1}}

\nc{\Beren}[1]{{\color{MidnightBlue}#1}}
\nc{\Drew}[1]{{\color{Orange}#1}}
\nc{\Tobi}[1]{{\color{Green}#1}}
\nc{\Natalia}[1]{{\color{Yellow}#1}}
\nc{\Dout}[1]{\Drew{\sout{#1}}}
\nc{\Bout}[1]{\Beren{\sout{#1}}}
\nc{\Tout}[1]{\Tobi{\sout{#1}}}
\nc{\Nout}[1]{\Natalia{\sout{#1}}}

\usepackage{todonotes}

\nc{\overbar}[1]{\mkern 1.5mu\overline{\mkern-1.5mu#1\mkern-1.5mu}\mkern 1.5mu}

\nc{\weaklyfinite}{weakly closed}
\nc{\finite}{closed}
\nc{\BCdual}[1]{{#1}^*}
\nc{\LCore}{\mathrm{LCore}}
\nc{\Stovicek}{\v{S}\v{t}ov\'{i}\v{c}ek}
\nc{\ftriple}{f_{\natural}}
\nc{\unitC}{\unit_{\cat C}}%
\nc{\unitD}{\unit_{\cat D}}%
\nc{\Pone}{{\mathbb{P}^1}}
\nc{\InvSupp}[1]{\Supp^{-1}(#1)}%
\nc{\InvCosupp}[1]{\Cosupp^{-1}(#1)}
\nc{\closureP}{\overbar{\{\cat P\}}}
\nc{\closureQ}{\overbar{\{\cat Q\}}}
\nc{\singP}{\{\cat P\}}
\nc{\singQ}{\{\cat Q\}}
\nc{\singm}{\{\mathfrak m\}}
\dmo{\Inj}{Inj}
\dmo{\Rep}{Rep}
\dmo{\Res}{Res}
\dmo{\KInjdmo}{K}
\dmo{\Dbdmo}{mod}
\nc{\KInj}[1]{\KInjdmo(\Inj #1)}
\nc{\Dbmod}[1]{\Der^b(\Dbdmo #1)}
\dmo{\Viss}{vis}%
\nc{\Vis}{\Viss}
\nc{\vis}{\Vis}
\nc{\kappaaux}{g}
\nc{\kappaCh}{{\kappaaux(\cat C_h)}}
\nc{\kappam}{{\kappaaux({\mathfrak m})}}
\nc{\kappaP}{{\kappaaux(\cat P)}}
\nc{\kappaQ}{{\kappaaux(\cat Q)}}
\nc{\kappaCP}{{\kappaaux_{\cat C}(\cat P)}}
\nc{\kappaDP}{{\kappaaux_{\cat D}(\cat P)}}
\nc{\kappaCQ}{{\kappaaux_{\cat C}(\cat Q)}}
\nc{\kappaDQ}{{\kappaaux_{\cat D}(\cat Q)}}
\nc{\kappaphiB}{{\kappaaux(\phi(\cat B))}}
\nc{\kappaphiQ}{{\kappaaux(\varphi(\cat Q))}}

\dmo{\Sub}{Sub}
\nc{\SpEn}{\cat S_{E(n)}}
\nc{\SpEnf}{\cat S_n}
\nc{\Lcomp}{L^{\mathrm{com}}} %
\nc{\Ucomp}{U^{\mathrm{com}}}
\nc{\bbullet}{{\scriptscriptstyle\hspace{-1pt}\bullet}}
\nc{\bullett}{{\scriptscriptstyle\bullet}\hspace{-1pt}}
\nc{\LF}{L\hspace{-0.2ex}F}
\dmo{\StMod}{StMod}
\dmo{\Proj}{Proj}
\dmo{\Ind}{Ind}
\nc{\SpG}{\Sp^G}
\nc{\EG}{\bbE_G}
\nc{\DEG}{\Der(\EG)}
\nc{\DE}{\Der(\bbE)}
\nc{\Prst}{{\cat P}\mathrm{r^{st}}}
\nc{\Mack}[2]{\mathrm{Mack}_{#1}(#2)}
\nc{\SC}{S\cat C}
\dmo{\fin}{{fin}}
\dmo{\DM}{DM}
\dmo{\fp}{fp}
\nc{\DMQ}{\DM_Q}

\dmo{\DerKal}{DMack}
\dmo{\Perf}{Perf}
\dmo{\coh}{coh}
\dmo{\Der}{D}
\dmo{\DMot}{DMot}
\dmo{\rmH}{H}
\dmo{\piu}{\underline{\pi}}
\dmo{\Sphere}{\mathbb{S}}
\nc{\HA}{{\rmH \hspace{-0.2em}\bbA}}
\nc{\HZ}{{\rmH \hspace{-0.2em}\bbZ}}
\nc{\HZbar}{{\rmH \hspace{-0.2em}\underline{\bbZ}}}
\nc{\Fp}{{\bbF_{\hspace{-0.1em}p}}}
\nc{\HFp}{{\rmH \hspace{-0.15em}\bbF_{\hspace{-0.1em}p}}}
\nc{\DHZG}{\Der(\HZ_G)}
\nc{\DHZH}{\Der(\HZ_H)}
\nc{\DHZK}{\Der(\HZ_K)}
\nc{\DHZGN}{\Der(\HZ_{G/N})}
\nc{\DHZGG}{\Der(\HZ_{G/G})}
\nc{\DHZCp}{\Der(\HZ_{C_p})}
\nc{\DHZGprime}{\Der(\HZ_{G'})}
\nc{\DHZ}{\Der(\HZ)}
\nc{\mathfrakp}{\mathfrak{p}}
\nc{\mathfrakq}{\mathfrak{q}}
\nc{\mathfrakS}{\mathfrak{S}}
\nc{\mathfrakT}{\mathfrak{T}}
\nc{\Z}{\mathbb{Z}}
\nc{\SSG}{\text{sSet}_*^G}
\nc{\sSet}{\text{sSet}}

\dmo{\csupp}{csupp}
\dmo{\Con}{Conj}
\dmo{\Id}{Id}
\dmo{\rmK}{\textrm{\rm K}}
\dmo{\Spc}{Spc}
\dmo{\thick}{thick}
\dmo{\thickid}{thickid}
\nc{\thicko}[1]{\thickid\langle #1 \rangle}
\nc{\thickt}[1]{\thick_\otimes\langle #1 \rangle}
\dmo{\cone}{cone}
\dmo{\End}{End}
\dmo{\Derperf}{D_{perf}}
\dmo{\Mor}{Mor}
\dmo{\id}{id}
\dmo{\incl}{incl}
\dmo{\Img}{Im}
\dmo{\im}{im}
\dmo{\Ker}{Ker}
\dmo{\ind}{ind}
\dmo{\CoInd}{coind}
\dmo{\res}{res}
\dmo{\infl}{infl}
\dmo{\Derqc}{D_{qc}}
\nc{\DbcohX}{\Der^b(\coh X)}
\dmo{\triv}{triv}
\dmo{\Tel}{Tel} %
\dmo{\grMod}{grMod}%
\dmo{\Mod}{Mod}%
\dmo{\opname}{op}
\dmo{\SH}{SH}%
\dmo{\smallb}{b}%
\dmo{\Spec}{Spec}
\dmo{\supp}{supp}
\dmo{\Supp}{Supp}
\dmo{\cosupp}{cosupp}
\dmo{\Cosupp}{Cosupp}
\nc{\SHc}{{\SH^c}}
\nc{\SHp}{{\SH_{(p)}}}
\nc{\SHcp}{{\SH^c_{(p)}}}
\nc{\SHG}{\SH(G)}
\nc{\SHGp}{\SH(G)_{(p)}}
\nc{\SHGc}{\SHG^c}
\nc{\SHGcp}{\SHG^c_{(p)}}
\nc{\quadtext}[1]{\quad\textrm{#1}\quad}
\nc{\qquadtext}[1]{\qquad\textrm{#1}\qquad}
\nc{\adj}{\dashv}
\nc{\adjto}{\rightleftarrows}
\nc{\bbL}{\mathbb{L}}
\nc{\bbA}{\mathbb{A}}
\nc{\bbE}{\mathbb{E}}
\nc{\bbN}{\mathbb{N}}
\nc{\bbQ}{\mathbb{Q}}
\nc{\bbZ}{\mathbb{Z}}
\nc{\bbF}{\mathbb{F}}
\nc{\bbT}{\mathbb{T}}
\nc{\cat}[1]{\mathscr{#1}}%
\nc{\ie}{{\sl i.e.}, }
\nc{\into}{\mathop{\rightarrowtail}}
\nc{\inv}{^{-1}}
\nc{\isoto}{\mathop{\overset{\sim}\to}}
\nc{\isotoo}{\mathop{\overset{\sim}\too}}
\nc{\onto}{\mathop{\twoheadrightarrow}}
\nc{\too}{\mathop{\longrightarrow}\limits}
\nc{\mapstoo}{\longmapsto}
\nc{\adh}[1]{\overline{#1}}%
\nc{\adhpt}[1]{\adh{\{#1\}}}%
\nc{\aka}{{a.\,k.\,a.}\ }
\nc{\calF}{\mathcal{F}}
\nc{\eg}{{\sl e.\,g.}}
\nc{\hook}{\hookrightarrow}
\nc{\borel}[2]{b_{#1}{#2}}
\nc{\ideal}[1]{\langle #1\rangle}

\dmo{\Hom}{Hom}
\nc{\Homcat}[1]{\Hom_{\cat #1}}
\nc{\iHom}{\mathcal{H}\mathrm{om}}
\nc{\ihom}[1]{\mathsf{hom}(#1)}
\nc{\ihomC}[1]{\mathsf{hom}_{\cat C}(#1)}
\nc{\ihomD}[1]{\mathsf{hom}_{\cat D}(#1)}
\nc{\ihomsub}[2]{\mathsf{hom}_{#1}(#2)}

\nc{\Mid}{\,\big|\,}
\nc{\MMod}{\,\text{-}\Mod}%
\nc{\GrMMod}{\,\text{-}\grMod}%
\nc{\op}{^{\opname}}
\nc{\oto}[1]{\overset{#1}\to}
\nc{\otoo}[1]{\overset{#1}{\,\too\,}}
\nc{\sminus}{\!\smallsetminus\!}
\nc{\poplus}[1]{^{\oplus #1}}%
\nc{\potimes}[1]{^{\otimes #1}}%
\nc{\sbull}{{\scriptscriptstyle\bullet}}%
\nc{\SET}[2]{\big\{\,#1\Mid#2\,\big\}}
\nc{\SpcK}{\Spc(\cat K)}%
\nc{\then}{\Rightarrow}
\nc{\unit}{\mathbb{1}}%
\nc{\xra}{\xrightarrow}
\nc{\phigeom}[1]{\widetilde{\Phi}^{#1}}
\dmo{\Oname}{O}
\dmo{\proper}{proper}%
\dmo{\lenormal}{\unlhd}
\dmo{\lnormal}{\lhd}
\nc{\normal}{\trianglelefteq}%
\nc{\Op}{\Oname^p}%
\nc{\Oq}{\Oname^q}%
\dmo{\Sp}{Sp}
\dmo{\Ho}{Ho}
\dmo{\Fin}{Fin}
\dmo{\add}{add}
\dmo{\Fun}{Fun}
\dmo{\Ext}{Ext}
\dmo{\CAlg}{CAlg}
\dmo{\CMon}{CMon}
\dmo{\CC}{\cat C} %
\dmo{\DD}{\cat D}
\dmo{\OO}{\mathcal{O}}
\dmo{\Map}{Map}
\dmo{\Span}{Span}
\dmo{\N}{N}
\dmo{\Cat}{Cat}
\dmo{\colim}{colim}
\dmo{\hocolim}{hocolim}
\dmo{\Ch}{Ch}
\dmo{\A}{\mathbb{A}^{eff}}
\nc{\AGeff}{\mathbb{A}_G^{\mathrm{eff}}}
\nc{\BGeff}{\mathcal{B}_G^{\mathrm{eff}}}
\nc{\BG}{{\mathcal{B}_G}}
\nc{\NBGeff}{{\N}{\BGeff}}
\dmo{\Ab}{Ab}
\dmo{\Set}{Set}
\dmo{\ev}{ev}
\dmo{\Spcl}{Spcl}
\nc{\Funadd}{\Fun_{\add}}
\dmo{\proj}{proj}
\dmo{\cof}{cof}

\dmo{\Coideal}{Coideal}
\dmo{\gen}{gen}

\nc{\auxcoidealsymb}{\vartriangleleft}

\dmo{\Loc}{Loc}
\dmo{\Coloc}{Coloc}
\dmo{\Locideal}{Locid}
\dmo{\Colocideal}{Colocid}
\nc{\LOCO}{\Locideal}
\nc{\COLOCO}{\Colocideal}
\nc{\Loco}[1]{\LOCO\langle #1 \rangle}
\nc{\Coloco}[1]{\COLOCO\langle #1 \rangle}
\nc{\Thickidset}[1]{\mathcal{THICK}_{\otimes}(#1)}
\nc{\Locidset}[1]{\mathcal{LOC}_{\otimes}(#1)}
\nc{\Colocidset}[1]{\mathcal{C}\mathrm{ocolid}(#1)}

\nc{\LambdaP}{\Lambda^{\cat P}} %
\nc{\LambdaQ}{\Lambda^{\cat Q}} %
\nc{\GammaP}{\Gamma_{\cat P}} %
\nc{\GammaQ}{\Gamma_{\cat Q}} %
\nc{\LambdaW}{\Lambda^{\hspace{-0.3ex}W}} %
\nc{\GammaW}{\Gamma_{\hspace{-0.3ex}W}} %

\nc{\gW}{g_W}
\nc{\gP}{g_{\cat P}}
\nc{\gQ}{g_{\cat Q}}
\nc{\cC}{{\cat C}}
\nc{\cT}{{\cat T}}
\nc{\cD}{{\cat D}}

\nc{\mT}{\kern-0.5em\mod\kern-0.1em\text{-}\cat{T}^c}
\nc{\mTc}{\kern-0.5em\mod\kern-0.1em\text{-}\cat{T}^c}
\nc{\MTc}{\Mod\kern-0.1em\text{-}\cat{T}^c}
\nc{\MT}{\Mod\kern-0.1em\text{-}\cat{T}}
\newcounter{enum-resume-hack}

\Crefname{Thm}{Theorem}{Theorems}
\Crefname{Prop}{Proposition}{Propositions}
\Crefname{thmx}{Theorem}{Theorems}

\begin{document}

\title[On surjectivity in tensor triangular geometry]{\mbox{On surjectivity in tensor triangular geometry}}

\author{Tobias Barthel}\thanks{TB is supported by the European Research Council (ERC) under Horizon Europe (grant No.~101042990). NC is partially supported by Spanish State Research Agency project PID2020-116481GB-I00, the Severo Ochoa and María de Maeztu Program for Centers and Units of Excellence in R$\&$D (CEX2020-001084-M), and the CERCA Programme/Generalitat de Catalunya. DH is supported by grant number TMS2020TMT02 from the Trond Mohn Foundation. BS is supported by NSF grant~DMS-1903429.}
\author{Nat{\`a}lia Castellana}
\author{Drew Heard}
\author{Beren Sanders}

\makeatletter
\patchcmd{\@setaddresses}{\indent}{\noindent}{}{}
\patchcmd{\@setaddresses}{\indent}{\noindent}{}{}
\patchcmd{\@setaddresses}{\indent}{\noindent}{}{}
\patchcmd{\@setaddresses}{\indent}{\noindent}{}{}
\makeatother

\address{Tobias Barthel, Max Planck Institute for Mathematics, Vivatsgasse 7, 53111 Bonn, Germany}
\email{tbarthel@mpim-bonn.mpg.de}
\urladdr{https://sites.google.com/view/tobiasbarthel/home}

\address{Nat{\`a}lia Castellana, Departament de Matem\`atiques, Universitat Aut\`onoma de Barcelona, 08193 Bellaterra, Spain, and Centre de Recerca Matemàtica}
\email{Natalia.Castellana@uab.cat}
\urladdr{http://mat.uab.cat/$\sim$natalia}

\address{Drew Heard, Department of Mathematical Sciences, Norwegian University of Science and Technology, Trondheim}
\email{drew.k.heard@ntnu.no}
\urladdr{https://folk.ntnu.no/drewkh/}

\address{Beren Sanders, Mathematics Department, UC Santa Cruz, 95064 CA, USA}
\email{beren@ucsc.edu}
\urladdr{http://people.ucsc.edu/$\sim$beren/}

\begin{abstract}
	We prove that a jointly conservative family of geometric functors between rigidly-compactly generated tensor triangulated categories induces a surjective map on Balmer spectra. From this we deduce a fiberwise criterion for Balmer's comparison map to be a continuous bijection. This gives short alternative proofs of the Hopkins--Neeman theorem and its generalization, due to Lau, to the case of a finite group acting trivially on an affine scheme.
\end{abstract}


\maketitle

\renewcommand{\thepart}{\Roman{part}}
\setcounter{section}{1} %

\subsection*{Introduction}

Tensor triangular geometry studies the global and local structure of tensor triangulated (`$tt$') categories through their Balmer spectra 
akin to the way algebraic geometry views commutative rings geometrically through their Zariski spectra. Understanding the Balmer spectrum is thus a fundamental problem wherever such categories arise including commutative algebra, algebraic and geometric topology, algebraic geometry, non-commutative geometry, Lie theory, and representation theory \cite{BalmerICM}.

A standard way to get a handle on the spectrum of a given tt-category $\cat K$ is to construct a tensor triangulated functor $\cat K \to \cat L$ to a tt-category whose spectrum is understood and then consider the effect of the induced map $\varphi\colon\Spc(\cat L)\to \Spc(\cat K)$.
Necessary and sufficient conditions for the map $\varphi$ to be surjective have been provided by Balmer \cite{Balmer18} and in such favourable circumstances one is able to capture all the points of $\Spc(\cat K)$ from the known points of $\Spc(\cat L)$.

In this paper, we provide an alternative characterization of surjectivity in tensor triangular geometry. In fact, we will characterize when an arbitrary family of functors $\cat K \to \cat L_i$ induces a jointly surjective family of maps $\Spc(\cat L_i)\to\Spc(\cat K)$.
The passage from a single functor to a family of functors provides considerably more flexibility as it greatly expands the collection of ``known'' categories one can utilize. As we will discuss below, this extension to an arbitrary family is not a mere formality and relies on recent advances in tt-geometry.
It enables us to derive a fiberwise criterion for Balmer's comparison map to the Zariski spectrum to be a bijection.

To demonstrate the utility of our methods, %
we give short and arguably simplified proofs of the Hopkins--Neeman theorem as well as Lau's recent computation of the Balmer spectrum of certain Deligne--Mumford stacks. Further applications can be found in \cite{BBB2024pp,Litterell2024,Hyslop2024}.

\subsection*{Results and proofs}

Throughout this note, we work in the context of rigidly-compactly generated tensor triangulated categories, usually denoted by $\cat S$ or $\cat T$. We write $\Spc(\cat T^c)$ for the associated Balmer spectrum of compact (=dualizable) objects and freely use basic constructions from tt-geometry \cite{Balmer05a,Balmer10b}. A coproduct-preserving tensor triangulated functor $f^*\colon\cat T\to \cat S$ is called a \emph{geometric functor}. Such a functor preserves compact objects and thus induces a continuous map $\varphi\colon \Spc(\cat S^c)\to \Spc(\cat T^c)$. Following terminology introduced in \cite{Balmer20_bigsupport,CarmeliSchlankYanovski2022}, a \emph{weak ring} in $\cat T$ is an object~$R \in \cat T$ equipped with a map $\eta \colon \unit \to R$ from the unit object such that the induced map $R\otimes \eta\colon R \to R\otimes R$ is a split monomorphism.

\begin{Def}
Suppose $\{f_i^*\colon \cat T \to \cat S_i\}_{i\in I}$ is a family of geometric functors between rigidly-compactly generated tt-categories. We say the family is
    \begin{itemize}
        \item \emph{jointly conservative} if for any $t \in \cat T$, $f_i^*(t) = 0$ for all $i \in I$  implies $t=0$;
        \item \emph{jointly nil-conservative} if for any weak ring $R \in \cat T$, $f_i^*(R) = 0$ for all $i \in I$ implies $R=0$.
    \end{itemize}
	Note that any jointly conservative family is in particular jointly nil-conservative. The converse does not hold:
\end{Def}
\begin{Exa}
	The Morava $K$-theories $\{K(n)\otimes-\colon \Sp \to \Mod(K(n))\}_{n \in \bbN\cup \{\infty\}}$ are jointly nil-conservative as a consequence of the nilpotence theorem \cite[Theorem~3]{HopkinsSmith98} but they are not jointly conservative since they all annihilate the Brown--Comenetz dual of the sphere \cite[Corollary B.12]{HoveyStrickland99}.
\end{Exa}
\begin{Conv}\label{conv:coproduct}
    Throughout this paper, coproducts are taken in the category of topological spaces, as opposed to the category of spectral spaces. Note that the underlying set of the latter is typically much larger than that of the former: for example, the $\bbN$-indexed coproduct of singletons may be identified with $\beta\bbN$, the Stone--\v{C}ech compactification of $\bbN$; see \cite[Example 10.1.5]{DickmannSchwartzTressl19}.
\end{Conv}
\begin{Thm}\label{thm:main}
	If $\{f_i^*\colon \cat T \to \cat S_i\}_{i\in I}$ is a jointly nil-conservative family of geometric functors, then the induced map
	\begin{equation}\label{eq:varphi}
		\varphi\colon \bigsqcup_{i\in I} \Spc(\cat S_i^c) \to \Spc(\cat T^c)
	\end{equation}
	is surjective.%
\end{Thm}

\begin{Rem}
	If the family is finite, then this result can be deduced from the criterion \cite[Theorem 1.3]{Balmer18} by first proving that the geometric functor $\prod_if_i^*\colon \cat T \to \prod_{i\in I} \cat S_i$ detects tensor nilpotence of morphisms with dualizable source, as in \cite[Section 2.3]{BCHNP1}. For an infinite family, such an argument cannot work directly. Firstly, as noted in \cite[Remark 2.7]{gomez}, an infinite product of rigidly compactly-generated categories is almost never rigidly compactly-generated with respect to the coordinatewise tensor-triangulated structure. Even if one works with the spectrum of dualizable objects, where one has
	$(\prod_{i \in I}\cat S_i)^d = \prod_{i \in I}\cat S_i^d$, we see that $\Spc(\prod_{i \in I}\cat S_i^d) \neq \bigsqcup_{i \in I}\Spc(\cat S_i^d)$ whenever infinitely many of the~$\cat S_i$ are nonzero. 
 Indeed, the spectrum $\Spc(\prod_{i \in I}\cat S_i^d)$ is a spectral space and in particular quasi-compact, while an infinite coproduct of non-empty spaces cannot be quasi-compact. 
 The same problem arises if we use underlying models and take the product in the $\infty$-category of compactly generated $\infty$-categories, which gives the ind-completion of $\prod_{i \in I}\cat S_i^c$. See \cite[Example~2.6]{gomez}.
\end{Rem}

\begin{Exa}
	If $\{k_i\}_{i \in I}$ is a (not necessarily finite) family of fields, then the Zariski spectrum $\Spec(\prod_{i \in I}k_i)$ is homeomorphic to the Stone--\v{C}ech compactification of the discrete indexing set $I$.
\end{Exa}

\begin{Rem}
	Using the Balmer--Favi support \cite{BalmerFavi11} and the techniques of \cite{BCHS1}, it is possible to prove \cref{thm:main} for arbitrary indexing sets $I$ under the additional hypothesis that $\Spc(\cat T^c)$ is weakly noetherian. However, since the construction of a surjective map as in \eqref{eq:varphi} is often the first step in understanding $\Spc(\cat T^c)$, making any assumption on its topology is not desirable. Consequently, our proof relies on a suitable support theory for big objects which exists unconditionally without any point-set assumptions on $\Spc(\cat T^c)$. Such a theory is provided by the homological residue fields developed in \cite{BalmerKrauseStevenson19,Balmer20_nilpotence, Balmer20_bigsupport}, from which we will draw freely. Indeed, we will derive \cref{thm:main} as a corollary of the following more complete statement:
\end{Rem}

\begin{Thm}\label{thm:main-h}
	A family $\{f_i^*\colon\cat T \to \cat S_i\}_{i \in I}$ of geometric functors is jointly nil-conservative if and only if the induced map on homological spectra
	\begin{equation}\label{eq:homologicalsurjectivity}
		\varphi^h\colon\bigsqcup_{i \in I} \Spc^h(\cat S_i^c) \to \Spc^h(\cat T^c)
	\end{equation}
	is surjective.
\end{Thm}

\begin{proof}
	$(\Rightarrow)$: Let $\cat B \in \Spc^h(\cat T^c)$ be a homological prime and consider the associated weak ring $E_{\cat B} \neq 0$; see \cite[Section 3]{BalmerKrauseStevenson19}. By assumption, there exists some $i \in I$ such that $f_i^*(E_{\cat B}) \neq 0$. For simplicity, write $f^* \coloneqq f_i^*$ and $f_*$ for its right adjoint. By the unit-counit identity a right adjoint is conservative on the essential image of its left adjoint, and so $f_*f^*(E_{\cat B}) \ne 0$.  Along with the projection formula \cite[(2.16)]{BalmerDellAmbrogioSanders16}, we deduce 
	\[
		f_*(\unit) \otimes E_{\cat B} \simeq f_*f^*(E_{\cat B}) \neq 0.
	\]
	Note that as a right adjoint to a tt-functor, $f_*$ is lax symmetric monoidal, hence~$f_*(\unit)$ is a weak ring in $\cat T$. Since the homological support coincides with the naive homological support for weak rings \cite[Theorem 4.7]{Balmer20_bigsupport}, this implies that $\cat B \in \Supp^h(f_*(\unit))$. By \cite[Theorem 5.12]{Balmer20_bigsupport}, we conclude that
	\[
		\cat B \in \Supp^h(f_*(\unit)) = \im(\Spc^h(f^*)),
	\]
	thereby verifying  that \eqref{eq:homologicalsurjectivity} is surjective.

	$(\Leftarrow)$: If $R \in \cat T$ is a nonzero weak ring, then $\Supp^h(R) \neq \emptyset$ by \cite[Thm.~1.8]{Balmer20_bigsupport}. Hence, if \eqref{eq:homologicalsurjectivity} is surjective then there exists an $i \in I$ such that 
	\[
		\im(\Spc^h(f_i^*)) \cap \Supp^h(R) \neq \emptyset.
	\]
	By \cite[Theorem 1.2(d) and Theorem 1.9]{Balmer20_bigsupport}, this implies $\Supp^h(f_{i*}f_i^*(R))= \Supp^h(f_{i*}(\unit) \otimes R) = \Supp^h(f_{i*}(\unit))\cap \Supp^h(R) \neq \emptyset$ so that $f_i^*(R) \neq 0$.
\end{proof}

\begin{proof}[Proof of \cref{thm:main}]
	In order to deduce \cref{thm:main} from \cref{thm:main-h}, we employ the naturality of the homological comparison map $\phi$ from \cite[Theorem 5.10]{Balmer20_bigsupport}, resulting in a commutative square:
	\[
		\xymatrix{\bigsqcup_{i \in I} \Spc^h(\cat S_i^c) \ar[r]^-{\varphi^h} \ar[d]_{\sqcup\phi_{\cat S_i}} & \Spc^h(\cat T^c) \ar[d]^-{\phi_{\cat T}}\\
		\bigsqcup_{i \in I} \Spc(\cat S_i^c) \ar[r]_-{\varphi} & \Spc(\cat T^c).}
	\]
	By \cite[Corollary 3.9]{Balmer20_nilpotence}, the vertical maps are surjective, and so is the top horizontal map by \cref{thm:main-h}. It follows that $\varphi$ is also surjective.
\end{proof}

\begin{Rem}
	It is an open question whether the converse to \cref{thm:main} holds, that is, whether the surjectivity of $\varphi$ in \eqref{eq:varphi} implies that the family $\{f_i^*\}_{i \in I}$ is jointly nil-conservative. It is known that the family need not be jointly conservative (see \cite[Example 14.26]{BCHS1}). In light of \cref{thm:main-h}, the converse of \cref{thm:main} would follow from Balmer's ``Nerves of Steel'' Conjecture that the homological and tensor triangular spectra always coincide; see \cite{bhs2}.
\end{Rem}

\begin{Rem}
	While \cref{thm:main} is in general not enough to determine the topology on $\Spc(\cat T^c)$ even when $\varphi$ is a bijection (see for instance \cite[Remark 15.12]{bhs1}), there are cases in which it can be used to compute the topology. Recall that Balmer \cite{Balmer10b} constructs a natural \emph{comparison map}
	\[
		\rho_{\cat T}\colon \Spc(\cat T^c) \to \Spec^h(\End_{\cat T}^*(\unit))
	\]
	between the tensor triangular spectrum and the Zariski spectrum of the graded endomorphism ring of the unit object. If $\cat T$ is \emph{noetherian} in the sense that $\End_{\cat T}^*(C)$ is noetherian as an $\End_{\cat T}^*(\unit)$-module for each $C \in \cat T^c$, then $\rho_{\cat T}$ is a homeomorphism if and only if it is a bijection \cite[Corollary 2.8]{Lau2021Balmer}. The following result provides a `fiberwise' criterion for Balmer's comparison map to be a continuous bijection:
\end{Rem}

\begin{Cor}\label{cor:fiberwise}
	Let $\cat T$ be a rigidly-compactly generated tt-category and consider a family of geometric tt-functors $\{f_i^*\colon \cat T \to \cat S_i\}_{i \in I}$ satisfying the following properties:     
	\begin{enumerate}
		\item the family $\{f_i^*\}_{i \in I}$ is jointly nil-conservative;
        \item $\rho_{\cat S_i}$ is a bijection for all $i \in I$;
        \item the induced map on Zariski spectra 
			\[
				\bigsqcup_{i\in I}\Spec^h(\End_{\cat S_i}^*(\unit)) \to \Spec^h(\End_{\cat T}^*(\unit))
			\]
        is a bijection.
    \end{enumerate}
	Then $\rho_{\cat T}$ is a bijection. If in addition $\cat T$ is noetherian, then $\rho_{\cat T}$ is a homeomorphism. 
\end{Cor}

\begin{proof}
	Naturality of the comparison map yields a commutative diagram
	\[
	\begin{tikzcd}
		\bigsqcup_{i \in I} \Spc(\cat S_i^c) \ar[r,"\varphi"] \ar[d, "\sqcup \rho_{\cat S_i}"'] & \Spc(\cat T^c) \ar[d, "\rho_{\cat T}"] \\
		\bigsqcup_{i \in I} \Spec^h(\End_{\cat S_i}^*(\unit)) \ar[r] & \Spec^h(\End_{\cat T}^*(\unit)).
	\end{tikzcd}
	\]
	On the one hand, by assumption, both the left vertical and the bottom horizontal maps are bijections, so $\varphi$ has to be injective. On the other hand, \cref{thm:main} implies that $\varphi$ is also surjective, hence bijective. It follows that $\rho_{\cat T}$ is a bijection and thus a homeomorphism whenever $\cat T$ is noetherian.
\end{proof}

\begin{Rem}
	\Cref{cor:fiberwise} offers an alternative perspective on the Hopkins--Neeman theorem \cite{Hopkins87, Neeman92a} for commutative rings:
\end{Rem}

\begin{Exa}\label{exa:hopkinsneeman}
	Let $\Der(R)$ be the derived category of a noetherian commutative ring~$R$. For any prime ideal $\frak p \in \Spec(R)$, consider the residue field $\kappa(\frak p)$, constructed as the quotient field of $R/\frak p$, and write $f_{\frak p}^*\colon \Der(R) \to \Der(\kappa(\frak p))$ for the associated base-change functor. We claim that the family $\{f_{\frak p}^*\}_{\frak p \in \Spec(R)}$ satisfies the assumptions of \cref{cor:fiberwise}. Indeed, $(b)$ and $(c)$ are immediate: since $\kappa(\frak p)$ is a field, $\rho_{\Der(\kappa(\frak p))}$ is a bijection (between singletons), while $(c)$ holds by construction. Finally, the family $\{f_{\frak p}^*\}$ is jointly conservative by \cite[Lemma 2.12]{Neeman92a},\punctfootnote{Note that the proof of this lemma does not rely on the nilpotence theorem or the thick subcategory theorem for $\Der(R)$.} which verifies $(a)$. Therefore, the comparison map
	\[
		\rho_{\Der(R)}\colon \Spc(\Der(R)^c) \xrightarrow{\sim} \Spec(R)
	\]
    is a homeomorphism. 
    
    The extension to arbitrary commutative rings follows by absolute noetherian approximation as in Thomason's work \cite{Thomason97}; cf.~\cite[Lemma 2.12]{Lau2021Balmer}. \qed
\end{Exa}

\begin{Exa}\label{exa:lau}
	Let $G$ be a finite group and let $R$ be a noetherian commutative ring equipped with trivial $G$-action. We write $\Rep(G,R)$ for the tt-category of $R$-linear derived representations of $G$ introduced in \cite{barthel2021rep1}. This category is noetherian and rigidly-compactly generated with subcategory of compact objects given by $\Der^b(\mathrm{mod}(G,R))$, the bounded derived category of $R[G]$-modules whose underlying complex of $R$-modules is perfect. If $k$ is a field, then $\Rep(G,k)$ coincides with the homotopy category of unbounded chain complexes of injective $k[G]$-modules studied in \cite{BensonKrause2008Complexes}; for an extension of this homological model to coefficients in $R$, see \cite{BBIKP_stratification}.

	For any prime ideal $\frak p \in \Spec(R)$, there is a geometric fiber-functor 
	\[
		F_{\frak p}^*\colon \Rep(G,R) \to \Rep(G,\kappa(\frak p)),
	\]
	which is induced by base-change along the canonical map $R \to \kappa(\frak p)$. We claim that the family $\{F_{\frak p}^*\}_{\frak p \in \Spec(R)}$ satisfies the conditions of \cref{cor:fiberwise}. Indeed, the joint conservativity of the family is the content of \cite[Proposition 3.6]{BBIKP_stratification}, while $\rho_{\Der^b(k[G])}$ is a homeomorphism by \cite{BensonCarlsonRickard97} for any field $k$. It remains to verify condition $(c)$. To this end, note that the map on Zariski spectra induced by $F_{\frak p}^*$ identifies with the composite
	\[
		\Spec^h(H^*(G,\kappa(\frak p))) \xrightarrow{\sim} \Spec^h(H^*(G,R)\otimes_R \kappa(\frak p)) \to \Spec^h(H^*(G,R)). 
	\]
	The first map is a homeomorphism by \cite[Corollary 8.29]{Lau2021Balmer}; see also \cite[Corollary 5.6]{BIKP2022fibrewise}. Varying the second map over $\Spec(R)$ assembles into a bijection
	\[
		\bigsqcup_{\frak p \in \Spec(R)}\Spec^h(H^*(G,R)\otimes_R \kappa(\frak p)) \xrightarrow{\sim} \Spec^h(H^*(G,R)),
	\]
	which verifies $(c)$ of \cref{cor:fiberwise} for $\{F_{\frak p}^*\}_{\frak p \in \Spec(R)}$. We conclude that $\rho_{\Der^b(\mathrm{mod}(G,R))}$ is a homeomorphism.\qed
\end{Exa}

\begin{Rem}
	\Cref{exa:lau} recovers the main theorem of \cite{Lau2021Balmer} for rings equipped with trivial $G$-action --- modulo the straightforward reduction  from the general case to the case where $R$ is noetherian, as explained at the beginning of the proof of \cite[Theorem 11.1]{Lau2021Balmer}. We remark that the key input to our proof is the joint conservativity of the functors $\{F_{\frak p}^*\}$ on the `big' categories $\Rep(G,R)$ and emphasize that the proof of this does not rely on the stratification of $\Rep(G,k)$.
\end{Rem}

\begin{Rem}
    The previous example extends to any finite flat group scheme over a noetherian commutative ring; cf.~\cite{BBIKP_stratification}. The key input is the recent generalization of the Friedlander--Suslin theorem \cite{FriedlanderSuslin1997} due to van der Kallen \cite{vanderKallen2023}. 
\end{Rem}

\subsection*{Acknowledgements}

We readily thank Paul Balmer for helpful conversations, Julia Pevtsova for raising the question of whether `fiberwise phenomena' admit a general tt-geometric explanation, and Juan Omar G\'{o}mez for pointing out some results from \cite{gomez}. We also thank the anonymous referees for helpful comments on an earlier version of this paper. We would also like to thank the Max Planck Institute and the Hausdorff Research Institute for Mathematics for their hospitality in the context of the Trimester program \emph{Spectral Methods in Algebra, Geometry, and Topology} funded by the Deutsche Forschungsgemeinschaft (DFG, German Research Foundation) under Germany’s Excellence Strategy – EXC-2047/1 – 390685813.

\bibliographystyle{alphasort}\bibliography{bibliography}
\end{document}